\documentclass[12pt]{amsart}
\usepackage{amscd, amssymb, pgf, tikz, hyperref}

\newcommand{\field}[1]{\mathbb{#1}}
\newcommand{\CC}{\field{C}}

\newcommand{\TT}{\field{T}}
\newcommand{\ZZ}{\field{Z}}
\newcommand{\RR}{\field{R}}

\newcommand{\Hh}{\mathcal{H}}

\newcommand{\la}{\langle}
\newcommand{\ra}{\rangle}
\newcommand{\ve}{\varepsilon}
\newcommand{\ds}{\displaystyle}

\newtheorem{thm}{Theorem}[section]
\newtheorem{cor}[thm]{Corollary}

\theoremstyle{definition}
\newtheorem{dfn}[thm]{Definition}

\theoremstyle{remark}
\newtheorem{rmk}[thm]{Remark}

\newtheorem{example}[thm]{Example}

\newtheorem*{examples*}{Examples}

\numberwithin{equation}{subsection}

\title{ Cuntz-Pimsner algebras of  group representations}

\author{Valentin Deaconu}
\address{Valentin Deaconu \\ Department of Mathematics (084)\\ University
of Nevada\\ Reno NV 89557-0084\\ USA} \email{vdeaconu@unr.edu}


\date{December 19, 2016}
\begin{document}
\begin{abstract}
Given a locally compact group $G$ and a unitary  representation $\rho:G\to U({\mathcal H})$ on a Hilbert space ${\mathcal H}$, we construct a $C^*$-correspondence ${\mathcal E}(\rho)={\mathcal H}\otimes_{\mathbb C} C^*(G)$ over $C^*(G)$ and study the Cuntz-Pimsner algebra ${\mathcal O}_{{\mathcal E}(\rho)}$. 
We prove that for $G$ compact,  ${\mathcal O}_{{\mathcal E}(\rho)}$ is strong Morita equivalent to a graph $C^*$-algebra.
If $\lambda$ is the left regular representation of an infinite, discrete and amenable group $G$, we show that ${\mathcal O}_{{\mathcal E}(\lambda)}$  is simple and purely infinite, with the same $K$-theory as $C^*(G)$.
If $G$ is compact abelian, any  representation decomposes into characters and determines a skew product graph.
We illustrate with several examples and we compare ${\mathcal E}(\rho)$ with the crossed product $C^*$-correspondence.

\end{abstract}
\maketitle
\section{introduction}

\bigskip

In a seminal paper \cite {K}, A. Kumjian studied the Cuntz-Pimsner algebra associated to a faithful representation $\pi$ of a separable unital $C^*$-algebra $A$ on a Hilbert space ${\mathcal H}$ such that $\pi(A)\cap {\mathcal K}({\mathcal H})=\{0\}$, where ${\mathcal K}({\mathcal H})$ denotes the set of compact operators. He considered the $C^*$-correspondence ${\mathcal E}={\mathcal H}\otimes_{\mathbb C}A$ over $A$ with natural structure and  proved that ${\mathcal O}_{\mathcal E}$ is simple and purely infinite. Moreover,   its isomorphism class depends only on the $K$-theory of $A$ and the class of the unit $[1_A]$. In \cite{BKP} the authors used this type of construction to prove that any order two automorphism of the $K$-theory of a unital Kirchberg algebra $A$ satisfying UCT with $[1_A]=0$ in $K_0(A)$ lifts to an order two automorphism of $A$. For more about lifting automorphisms of $K$-groups to Kirchberg algebras, see \cite {Ka08} and \cite{S}.

In this paper, we study a similar $C^*$-correspondence ${\mathcal E}(\rho)$ over $A=C^*(G)$, the $C^*$-algebra of a locally compact group $G$, where the representation $\pi$ of $A$ is obtained by integrating a unitary representation $\rho :G\to U({\mathcal H})$. In our setting, $\pi$ is not necessarily faithful and the intersection of $\pi(A)$ with the compact operators is not necessarily trivial. In particular, the left multiplication on ${\mathcal E}(\rho)$ may not be injective and the resulting Cuntz-Pimsner algebra may not be simple or purely infinite.
For compact groups, we prove that ${\mathcal O}_{\mathcal E}$ is strong Morita equivalent to a graph $C^*$-algebra. Since any representation decomposes into a direct sum of irreducible representations,  it suffices to study $C^*$-correspondences arising from these representations. We  illustrate how  basic operations on representations  reflect on the associated graphs. 

Graphs associated to representations of finite or compact groups   were already considered in the work of J. McKay \cite{Mc} and of M.H. Mann, I. Raeburn and C.E. Sutherland \cite{MRS}. Given a representation $\rho$ of  a  group $G$ with set of irreducible representations $\hat{G}=\{\pi_j\}$, they consider the graph with vertex set $\hat{G}$ and incidence matrix $[m_{jk}]$ where $\ds \rho\otimes \pi_j=\bigoplus_km_{jk}\pi_k$. This kind of graph has connections with the Doplicher-Roberts algebras, see \cite{MRS} and \cite{De}. The  graphs obtained from our ${\mathcal E}(\rho)$ are in general different from these graphs, since they may have sources.

 For $G$ infinite, discrete and amenable, it is known that the left regular representation $\lambda$ extends to a faithful representation of $C^*(G)$. Since $C^*(G)$ is unital and  the intersection with the compacts is trivial, it follows from \cite{K} that  the Cuntz-Pimsner algebra ${\mathcal O}_{{\mathcal E}(\lambda)}$ is simple and purely infinite, $KK$-equivalent with $C^*(G)$.  It is a challenge to understand the Cuntz-Pimsner algebra associated with any representation of an arbitray group. Sometimes there are connections with $C^*$-algebras of topological graphs, like in the case $G={\mathbb Z}$.

We compare the  $C^*$-correspondence ${\mathcal E}(\rho)$ with the crossed product $C^*$-correspondence ${\mathcal D}(\rho)$ associated to a group action (see \cite{De}). Recall that a representation of dimension $n\ge 2$ determines a quasi-free action on the Cuntz algebra ${\mathcal O}_n$ such that ${\mathcal O}_n\rtimes G\cong {\mathcal O}_{\mathcal D(\rho)}$. In the case $G$ is  abelian, we review some examples  studied by Kumjian and Kishimoto \cite{KK} and  by Katsura (see \cite{Ka03}), from a new viewpoint.

For more $C^*$-correspondences over group $C^*$-algebras, see also the recent preprint \cite{KLQ}.

\bigskip

\section{The $C^*$-correspondence of a group representations}
\bigskip

 Let $G$ be a (second countable) locally compact group. A unitary group representation is a homomorphism $\rho :G\to U({\mathcal H})$, where ${\mathcal H}$ is a (separable) Hilbert space such that for any fixed $\xi\in{\mathcal H}$ the map $g\to\rho(g)\xi$ is continuous. The left regular representation is

\[\lambda:G\to U(L^2(G)), \;\; \lambda(g)\xi(h)=\xi(g^{-1}h).\]

A representation $\rho:G\to U({\mathcal H})$ extends by integration
to \[\pi=\pi_\rho:C^*(G)\to{\mathcal L}({\mathcal H})\] such that 
\[\pi(f)\xi=\int_Gf(t)\rho(t)\xi dt\;  \;\text{for}\; f\in L^1(G), \xi\in {\mathcal H}.\]

\begin{dfn} The $C^*$-correspondence of a representation $\rho$ is  ${\mathcal E}={\mathcal E}(\rho)={\mathcal H}\otimes_{\mathbb C} C^*(G)$ where for $ \xi, \eta\in {\mathcal H}$ and $ a,b \in C^*(G)$  the  inner product is given by
\[\la\xi\otimes a,\eta\otimes b\ra=\la\xi,\eta\ra a^*b\]
and the right and left multiplications are
\[(\xi\otimes a)\cdot b=\xi\otimes ab,\;\;  a\cdot (\xi\otimes b)=\pi(a)\xi\otimes b.\] 
\end{dfn}
It is easy to see that equivalent representations determine isomorphic $C^*$-correspondences.

\begin{thm} If $G$ is a compact group and  $\rho: G\to U({\mathcal H})$  is any representation, then ${\mathcal O}_{{\mathcal E}(\rho)}$ is strong Morita equivalent (SME) to a graph $C^*$-algebra. If $\pi_\rho:C^*(G)\to{\mathcal L}({\mathcal H})$ is injective, then the graph has no sources. If $\rho\cong \rho_1\oplus \rho_2$, then the incidence matrix for the graph of $\rho$ is the sum of incidence matrices for $\rho_1$ and $\rho_2$. 

\end{thm}
 
\begin{proof} By the Peter-Weyl theorem, the group $C^*$-algebra $C^*(G)$ decomposes as a direct sum of matrix algebras $A_i=M_{n(i)}$ with units $p_i$, indexed by the discrete set $\hat{G}$ of equivalence classes of irreducible representations.

 Let $E$ be the graph with vertex space $E^0=\hat{G}$ and with edges joining the vertices $v_i$ and $v_j$ determined by the number of minimal components of the (non-zero) $A_j$-$A_i$ $C^*$-correspondences $p_j{\mathcal E}(\rho)p_i$. If $p_j{\mathcal E}(\rho)p_i=0$, there is no edge from $v_i$ to $v_j$. 
 
 It follows  that ${\mathcal O}_{{\mathcal E}(\rho)}$ is isomorphic to the $C^*$-algebra of the graph of $C^*$-correspondences in which we assign the algebra $A_i$ at the vertex $v_i$ and the minimal components of $p_j{\mathcal E}(\rho)p_i\neq 0$ for each edge joining $v_i$ with $v_j$. By construction, this $C^*$-algebra is SME to $C^*(E)$ (see \cite{MS} and \cite{KPQ}). For more on graphs of $C^*$-correspondences, see \cite{DKPS}.
 
 When $\pi_\rho:C^*(G)\to{\mathcal L}({\mathcal H})$ is injective, it follows that $p_j{\mathcal E}(\rho)\neq 0$,  and therefore  $v_j$ is not a source for all $j$. For the last part, notice that ${\mathcal E}(\rho)\cong {\mathcal E}(\rho_1)\oplus {\mathcal E}(\rho_2)$.
\end{proof}

We recall the following result of Kumjian, see Theorem 3.1 in \cite{K}:

\begin{thm}\label{spi}  Let $A$ be a separable unital $C^*$-algebra and let ${\mathcal E}={\mathcal H}\otimes_{\mathbb C}A$ with left multiplication given by a faithful representation $\pi:A\to {\mathcal L}({\mathcal H})$ such that $\pi(A)\cap{\mathcal K}({\mathcal H})=\{0\}$.
 Then  ${\mathcal O}_{\mathcal E}\cong {\mathcal T}_{\mathcal E}$ is simple, purely infinite and  $KK$-equivalent to $A$.
\end{thm}

\begin{cor} Let $G$ be infinite, discrete and amenable and let $\lambda:G\to U(\ell^2(G))$ be the left regular representation. Then ${\mathcal O}_{\mathcal E(\lambda)}$ is simple and purely infinite, $KK$-equivalent to $C^*(G)$.

\end{cor}
 \begin{proof} Since $G$ is discrete, $C^*(G)$ is unital. Since $G$ is amenable, it is known that the representation $\pi_\lambda:C^*(G)\to {\mathcal L}(\ell^2(G))$ induced by $\lambda$ is faithful (see for example Theorem A.18 in \cite{W}). Since $G$ is infinite, we also have \[\pi_\lambda(C^*(G))\cap{\mathcal K}(\ell^2(G))=\{0\},\]
 and we can apply the above theorem.
  \end{proof}
 Note that the same conclusion holds true for any representation $\rho$ of an infinite discrete group $G$ such that $\pi_\rho$ is faithful.

\bigskip
 
 \section{Examples}
 
 \bigskip

\begin{example} Let $S_3=\{(1),(12), (13),(23), (123), (132)\}$ be the symmetric group. Then $\hat{S_3}=\{\iota,\ve,\sigma\}$ where \[\iota:S_3\to U(\CC),\; \iota(g)=1\] is the trivial representation, \[\ve:S_3\to U(\CC), \; \ve(12)=-1,\; \ve(123)=1\] is the signature representation, and \[\sigma:S_3\to U(\CC^2),\;
\sigma(12)=\left[\begin{array}{rr}-1&-1\\0&1\end{array}\right],\;\;\sigma(123)=\left[\begin{array}{rr}-1&-1\\1&0\end{array}\right]\]
is the (standard) irreducible two-dimensional representation.
These representations  have characters \[\chi_\iota(g)=1,\; \chi_\ve(1)=1,\; \chi_\ve(12)=-1,\; \chi_\ve(123)=1\] and \[\chi_\sigma(1)=2, \; \chi_\sigma(12)=0,\; \chi_\sigma(123)=-1.\]

The group algebra $C^*(S_3)$ is isomorphic to $\CC\oplus\CC\oplus M_2(\CC)$ with unit $p_1\oplus p_2\oplus p_3=\ds(1/6)\chi_\iota\oplus(1/6)\chi_\ve\oplus(1/3)\chi_\sigma$. 

\medskip

Any representation $\rho:S_3\to U(\Hh)$ extends  to a representation \[\pi_\rho:C^*(S_3)\to {\mathcal L}({\Hh}), \pi_\rho(\sum a_g\delta_g)=\sum a_g\rho(g),\]  
where for $g\in S_3,$ $a_g\in \CC$, $\delta_g(h)=1$ for $h=g$ and $\delta_g(h)=0$ otherwise.

Since $\rho$ decomposes as a direct sum of irreducibles (with multiplicities), we first list  the graphs associated with the representations $\iota, \ve$ and $\sigma$. 

Since $\pi_\iota(p_1)=1, \pi_\iota(p_2)=\pi_\iota(p_3)=0$ we have 
${\mathcal E}(\iota)={\CC}\otimes C^*(S_3)\cong \CC\oplus \CC\oplus M_2(\CC)$
which gives the graph of $C^*$-correspondences

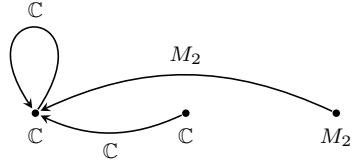
\[
\begin{tikzpicture}[shorten >=0.4pt,>=stealth, semithick]
\renewcommand{\ss}{\scriptstyle}
\node[inner sep=1.0pt, circle, fill=black]  (u) at (-2,0) {};
\node[below] at (u.south)  {$\ss \CC$};
\node[inner sep=1.0pt, circle, fill=black]  (v) at (0,0) {};
\node[below] at (v.south)  {$\ss \CC$};
\node[inner sep=1.0pt, circle, fill=black]  (w) at (2,0) {};
\node[below] at (w.south)  {$\ss M_2$};
\draw[->] (u) ..controls (-1,1.5) and (-3,1.5).. node[above,black] {$\ss \CC$} (u);

\draw[->] (w) to [out=155, in=25] node[above,black] {$\ss M_2$}  (u);
\draw[->] (v) to [out=-155, in=-25] node[below,black] {$\ss \CC$} (u);

\end{tikzpicture}
\]
Similarly, $\pi_{\ve}(p_2)=1, \pi_\ve(p_1)=\pi_\ve(p_3)=0$ and ${\mathcal E}(\ve)={\CC}\otimes C^*(S_3)\cong \CC\oplus \CC\oplus M_2(\CC)$ gives 

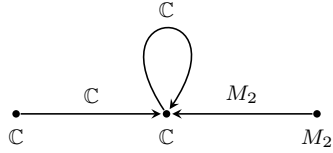
\[
\begin{tikzpicture}[shorten >=0.4pt,>=stealth, semithick]
\renewcommand{\ss}{\scriptstyle}
\node[inner sep=1.0pt, circle, fill=black]  (u) at (-2,0) {};
\node[below] at (u.south)  {$\ss \CC$};
\node[inner sep=1.0pt, circle, fill=black]  (v) at (0,0) {};
\node[below] at (v.south)  {$\ss \CC$};
\node[inner sep=1.0pt, circle, fill=black]  (w) at (2,0) {};
\node[below] at (w.south)  {$\ss M_2$};
\draw[->] (v) ..controls (-1,1.5) and (1,1.5).. node[above,black] {$\ss \CC$} (v);

\draw[->] (w) to  node[above,black] {$\ss M_2$}  (v);
\draw[->] (u) to  node[above,black] {$\ss \CC$} (v);

\end{tikzpicture}
\]
Since $\pi_\sigma(p_3)=\left[\begin{array}{cc}1&0\\0&1\end{array}\right]$ and $\pi_\sigma(p_1)=\pi_\sigma(p_2)=\left[\begin{array}{cc}0&0\\0&0\end{array}\right]$, it follows that ${\mathcal E}(\sigma)={\CC}^2\otimes C^*(S_3)\cong \CC^2\oplus \CC^2\oplus \CC^2\otimes M_2(\CC)$ determines 

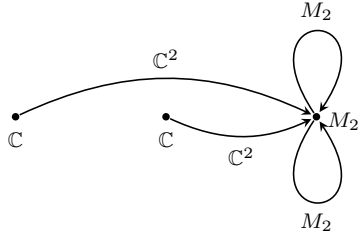
\[
\begin{tikzpicture}[shorten >=0.4pt,>=stealth, semithick]
\renewcommand{\ss}{\scriptstyle}
\node[inner sep=1.0pt, circle, fill=black]  (u) at (-2,0) {};
\node[below] at (u.south)  {$\ss \CC$};
\node[inner sep=1.0pt, circle, fill=black]  (v) at (0,0) {};
\node[below] at (v.south)  {$\ss \CC$};
\node[inner sep=1.0pt, circle, fill=black]  (w) at (2,0) {};
\node[right] at (w.south)  {$\ss M_2$};
\draw[->] (w) ..controls (1,1.5) and (3,1.5).. node[above,black] {$\ss M_2$} (w);
\draw[->] (w) ..controls (1,-1.5) and (3,-1.5).. node[below,black] {$\ss M_2$} (w);

\draw[->] (u) to [out=25, in=155] node[above,black] {$\ss \CC^2$}  (w);
\draw[->] (v) to [out=-25, in=-155] node[below,black] {$\ss \CC^2$} (w);
\end{tikzpicture}
\]

\end{example}

\begin{example}\label {perm}

 For the permutation representation $\rho:S_3\to U(\CC^3)$ we have $\pi_\rho\cong \pi_\iota\oplus\pi_\sigma$ and 
\[\pi_\rho(p_1)=\left[\begin{array}{ccc}1&0&0\\0&0&0\\0&0&0\end{array}\right], \;\; \pi_\rho(p_2)=0, \;\; \pi_\rho(p_3)=\left[\begin{array}{ccc}0&0&0\\0&1&0\\0&0&1\end{array}\right].\] 
The $C^*$-correspondence ${\mathcal E}(\rho)=\CC^3\otimes C^*(S_3)\cong \CC^3\oplus\CC^3\oplus \CC^3\otimes M_2$ decomposes as
\[{\mathcal E}(\rho)=p_1{\mathcal E}(\rho)p_1\oplus p_3{\mathcal E}(\rho)p_1\oplus p_1{\mathcal E}(\rho)p_2\oplus p_3{\mathcal E}(\rho)p_2\oplus p_1{\mathcal E}(\rho)p_3\oplus p_3{\mathcal E}(\rho)p_3\cong \]
\[\cong \CC\oplus\CC^2\oplus \CC\oplus\CC^2\oplus M_2\oplus \CC^2\otimes M_2.\]
Counting dimensions, we obtain the following graph of $C^*$-correspondences 

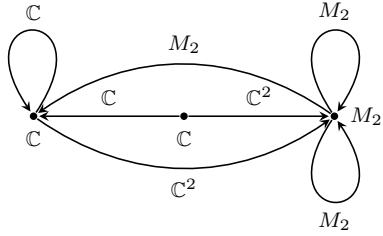
\[
\begin{tikzpicture}[shorten >=0.4pt,>=stealth, semithick]
\renewcommand{\ss}{\scriptstyle}
\node[inner sep=1.0pt, circle, fill=black]  (u) at (-2,0) {};
\node[below] at (u.south)  {$\ss \CC$};
\node[inner sep=1.0pt, circle, fill=black]  (v) at (0,0) {};
\node[below] at (v.south)  {$\ss \CC$};
\node[inner sep=1.0pt, circle, fill=black]  (w) at (2,0) {};
\node[right] at (w.east)  {$\ss M_2$};
\draw[->] (w) ..controls (1,1.5) and (3,1.5).. node[above,black] {$\ss M_2$} (w);
\draw[->] (w) ..controls (1,-1.5) and (3,-1.5).. node[below,black] {$\ss M_2$} (w);

\draw[->] (u) to [out=-35, in=-145] node[below,black] {$\ss \CC^2$}  (w);
\draw[->] (v) to node[above,black] {$\ss \CC^2$} (w);
\draw[->] (u) ..controls (-1,1.5) and (-3,1.5).. node[above,black] {$\ss \CC$} (u);

\draw[->] (w) to [out=145, in=35] node[above,black] {$\ss M_2$}  (u);
\draw[->] (v) to  node[above,black] {$\ss \CC$} (u);
\end{tikzpicture}
\]
The subjacent graph $E$ has incidence matrix \[B_\rho=B_\iota+B_\sigma=\left[\begin{array}{ccc}1&1&1\\0&0&0\\1&1&2\end{array}\right],\] 
where the entry $B_\rho(v,w)$ counts the edges from $w$ to $v$.
It follows that  ${\mathcal O}_{\mathcal E(\rho)}$ is SME to the  graph algebra $C^*(E)$.

Using Theorem 3.2 in \cite{RS}  we obtain
\[K_0(C^*(E))\cong\text{ coker }D\cong \ZZ, \;\; K_1(C^*(E))\cong \ker D\cong 0,\]
where \[D=\left[\begin{array}{rr}0&-1\\-1&0\\-1&-2\end{array}\right].\]
Compared with \cite{RS}, note that we reverse the direction of the edges, so sinks   become sources. 

\end{example}
\begin{rmk}
In all the above  cases the graphs have sources since the extension of the given representation to $C^*(S_3)$ is not one-to-one. 

If we consider a representation $\rho$ of $S_3$ which contains each of $\iota,\ve$ and $\sigma$, for example $\sigma\otimes \sigma=\iota\oplus\ve\oplus \sigma$ or the left regular representation $\lambda=\iota\oplus\ve\oplus 2\sigma$, then $\pi_\rho$ will be injective and the graph associated to ${\mathcal E}(\rho)$ will have no sources. Its incidence matrix is obtained by adding the incidence matrices corresponding to $\iota, \ve$ and $\sigma,$ counting multiplicities. 

For $\rho=\sigma\otimes \sigma$ we get the following graph of $C^*$-correspondences 
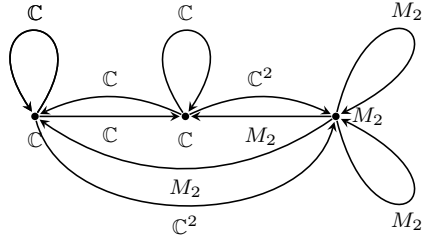
\[
\begin{tikzpicture}[shorten >=0.4pt,>=stealth, semithick]
\renewcommand{\ss}{\scriptstyle}
\node[inner sep=1.0pt, circle, fill=black]  (u) at (-2,0) {};
\node[below] at (u.south)  {$\ss \CC$};
\node[inner sep=1.0pt, circle, fill=black]  (v) at (0,0) {};
\node[below] at (v.south)  {$\ss \CC$};
\node[inner sep=1.0pt, circle, fill=black]  (w) at (2,0) {};
\node[right] at (w.east)  {$\ss M_2$};
\draw[->] (w) ..controls (2.5,2) and (4,1).. node[above,black] {$\ss M_2$} (w);
\draw[->] (w) ..controls (2.5,-2) and (4,-1).. node[below,black] {$\ss M_2$} (w);

\draw[->] (u) to [out=-75, in=-105] node[below,black] {$\ss \CC^2$}  (w);
\draw[->] (u) ..controls (-1,1.5) and (-3,1.5).. node[above,black] {$\ss \CC$} (u);
\draw[->] (v) ..controls (-1,1.5) and (1,1.5).. node[above,black] {$\ss \CC$} (v);

\draw[->] (u) ..controls (-1,1.5) and (-3,1.5).. node[above,black] {$\ss \CC$} (u);
\draw[->] (w) to  node[below,black] {$\ss M_2$}  (v);
\draw[->] (u) to  node[below,black] {$\ss \CC$} (v);
\draw[->] (v) to [out=25, in=155] node[above,black] {$\ss \CC^2$} (w);
\draw[->] (w) to [out=215, in=-35] node[below,black] {$\ss M_2$}  (u);
\draw[->] (v) to [out=155, in=25] node[above,black] {$\ss \CC$} (u);
\end{tikzpicture}
\]
with incidence matrix
\[B_{\sigma\otimes\sigma}=B_\iota+B_\ve+B_\sigma=\left[\begin{array}{ccc}1&1&1\\1&1&1\\1&1&2\end{array}\right].\] 

\end{rmk}

\begin{example} Any representation $\rho$ of a cyclic group $G$  is determined by a unitary $\rho(1)\in U({\mathcal H})$
and decomposes as a direct sum or a direct integral of characters.
Recall that this is just a restatement of the spectral theorem for unitary operators.

If $G={\mathbb Z}/n{\mathbb Z}$, then $\hat{G}=\{\chi_1,...,\chi_n\}$ and ${\mathcal E}(\rho)$ determines a graph with $n$ vertices where the incidence matrix $[a_{ji}]$ is such that $a_{ji}=\dim\chi_j{\mathcal H}\chi_i$. 

For $G={\mathbb Z}$, let's assume that ${\mathcal H}=L^2(X,\mu)$ for a measure space $(X,\mu)$ and that $\rho(1)=M_{\varphi}$, the multiplication operator with a  function $\varphi: X\to \TT$.

Then ${\mathcal E}(\rho) =L^2(X,\mu)\otimes C^*(\ZZ)\cong C(\TT, L^2(X,\mu))$ becomes a $C^*$-correspondence over $C^*(\ZZ)\cong C(\TT)$ with operations
\[\la \xi,\eta\ra(k)=\la \xi (k), \eta (k)\ra_{L^2(X,\mu)},\;\;\text{for}\;\;  \xi, \eta\in C_c(\ZZ, L^2(X,\mu))\]
and
\[(\xi\cdot f)(k)=\sum_k\xi(k)f(k),\;\;(f\cdot\xi)(k)=\sum_k f(k) \varphi^k\xi(k),\]\[ \;\;\text{for}\;\; f\in C_c(\ZZ), \xi\in C_c(\ZZ, L^2(X,\mu)).\]

\bigskip

If $\dim{\mathcal H}=n$ is finite, then the function $\varphi$ is given by  $(z_1,...,z_n)\in \TT^n$ and ${\mathcal E}(\rho)$ is isomorphic to the $C^*$-correspondence  of the  topological graph with vertex space $E^0={\mathbb T}$, edge space $E^1=\TT\times\{1,2,...,n\}$, source map $s:E^1\to E^0, s(z,k)=z$ and range map $r:E^1\to E^0, r(z,k)=z_kz$.

If $\lambda$ is the left regular representation of $G={\mathbb Z}$ on $\ell^2(\ZZ)$, it follows from Theorem \ref{spi} that ${\mathcal O}_{\mathcal E(\lambda)}$ is simple and purely infinite with the same $K$-theory as $C({\mathbb T})$.

\end{example}

\begin{example}
Let $G=\RR$ and let $\mu$ be the (normalized) Lebesgue measure $\mu$ on $\RR$. Consider the representation \[\rho:\RR\to U(L^2(\RR, \mu)), (\rho(t)\xi)(s)=e^{its}\xi(s)\] which extends to the Fourier transform $\pi_\rho$ on $C^*(\RR)\cong C_0(\RR)$, where 
\[(\pi_\rho(f)\xi)(s)=\int_\RR f(t)e^{its}\xi(s) d\mu(t),\;\text{ for}\; f\in L^1(\RR, \mu),\xi\in L^2(\RR, \mu).\] It is known that $\rho$ is equivalent to the right regular representation of $\RR$ (see Theorem 2.2 on page 117 in \cite{Su})
and that $\rho$ is a direct integral of characters $\chi_t$, where $\chi_t(s)=e^{its}$. 

Then ${\mathcal E}(\rho)=L^2(\RR, \mu)\otimes C^*(\RR)\cong C_0(\RR, L^2(\RR, \mu))$ becomes a $C^*$-correspondence over $C^*(\RR)\cong C_0(\RR)$ such that the left multiplication is injective and $\pi_\rho(C_0(\RR))\cap {\mathcal K}(L^2(\RR, \mu))=\{0\}$.
It follows that ${\mathcal O}_{\mathcal E(\rho)}$ has the $K$-theory of $C_0(\RR)$, but since $C_0(\RR)$ is not unital, we cannot apply Theorem \ref{spi} to conclude that this algebra is simple or purely infinite.
\end{example}

\section{The crossed product $C^*$-correspondence}

\bigskip

We want to compare our $C^*$-correspondence ${\mathcal E}(\rho)$ associated to a group representation with another $C^*$-correspondence from the literature.
Let $G$ be a locally compact group and let $\rho: G\to U({\mathcal H})$ be a representation   with $\dim{\mathcal H}=n$ for $n\in {\mathbb N}\cup\{\infty\}$. In \cite{De} we studied the crossed-product ${\mathcal O}_n\rtimes G$, using
the crossed product $C^*$-correspondence ${\mathcal D(\rho)}$ over $C^*(G)$, constructed also from ${\mathcal H}\otimes_{\CC} C^*(G)$ with the same inner product and right multiplication as ${\mathcal E}(\rho)$
\[\la \xi, \eta\ra(t)=\int_G \la\xi(s),\eta(st)\ra ds,\]
\[(\xi\cdot f)(t)=\int_G\xi(s)f(s^{-1}t)ds,\]
but with a different left multiplication
\[(h\cdot \xi)(t)=\int_Gh(s)\rho(s)\xi(s^{-1}t)ds\]
for $\xi, \eta\in C_c(G,\mathcal H)$ and $ f,h\in C_c(G)$.
This left multiplication is  always one-to-one, which changes the structure of the Cuntz-Pimsner algebra ${\mathcal O}_{\mathcal D(\rho)}$.
We recall the following  result (see \cite{HN}):

\begin{thm} The representation $\rho$ determines a quasi-free action of $G$ on the Cuntz algebra ${\mathcal O}_n$ and  ${\mathcal O}_{\mathcal D(\rho)}\cong {\mathcal O}_n\rtimes_\rho G$.

\end{thm}

\begin{example} If $G=S_3$ and $\rho: S_3\to U({\mathbb C}^3)$ is the permutation representation, then ${\mathcal D}(\rho)={\mathbb C}^3\otimes C^*(S_3)$ with the above operations becomes a $C^*$-correspondence over $C^*(S_3)$ different from ${\mathcal E}(\rho)$ discussed in Example \ref{perm}. It
decomposes as ${\mathbb C}\oplus {\mathbb C}\oplus M_2\oplus M_2\oplus{\mathbb C}^2\oplus{\mathbb C}^2\oplus{\mathbb C}^2\oplus{\mathbb C}^2$, see Example 6.4 in  \cite{De} and it determines the following graph of  $C^*$-correspondences for ${\mathcal O}_3\rtimes_\rho S_3$:

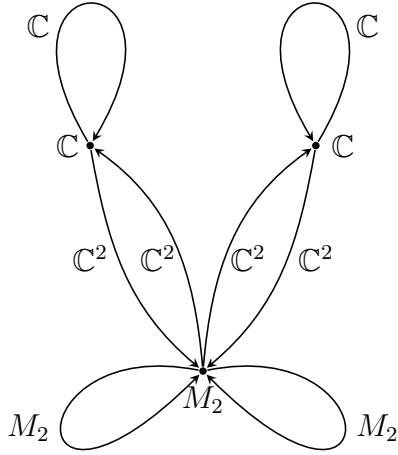
\[
\begin{tikzpicture}[shorten >=0.2pt,>=stealth, semithick]
\renewcommand{\ss}{\scriptstyle}
\node[inner sep=1.0pt, circle, fill=black]  (u) at (-1.5,2) {};
\node[left] at (u)  {$\mathbb C$};
\node[inner sep=1.0pt, circle, fill=black]  (v) at (1.5,2) {};
\node[right] at (v.east)  {$\mathbb C$};
\node[inner sep=1.0pt, circle, fill=black]  (w) at (0,-1) {};
\node[below] at (w.south)  {$ M_2$};
\draw[->] (w) to [out=95, in=-35] (u) node at (-1.5, 0.5)  {${\mathbb C}^2$};
\draw[->] (u) to [out=-80, in=140] (w) node at (-0.6, 0.5) {${\mathbb C}^2$};
\draw[->] (w) to [out=85, in=215]  (v) node at (0.6,0.5) {${\mathbb C}^2$};
\draw[->] (v) to [out=-100, in=40]  (w) node at (1.5, 0.5) {${\mathbb C}^2$};
\draw[->] (u) ..controls (-3,4.5) and (0,4.5)..(u) node[pos=0.3, left]{$\mathbb C$};
\draw[->] (v) ..controls (3,4.5) and (0,4.5)..(v) node[pos=0.3, right]{$\mathbb C$} ;
\draw[->] (w) ..controls (-2.5, -0.5) and (-2.5, -3.5) ..  (w)node[pos=0.5,left] {$ M_2$};
\draw[->] (w) ..controls (2.5,-0.5) and (2.5, -3.5).. (w)node[pos=0.5,right] {$M_2$} ;
\end{tikzpicture}
\]

 The subjacent graph has no sources and the incidence matrix  is

\[\left[\begin{array}{ccc}1&0&1\\0&1&1\\1&1&2\end{array}\right]\]
\end{example}

\bigskip

\section{Quasi-free actions of abelian groups}

\bigskip

If $G$ is  compact and abelian, then any representation $\rho$ of dimension $n\ge 2$ (including $n=\infty$) decomposes into characters and determines a cocycle  $c:E_n^1\to \hat{G}$, where $E_n$ is the graph with one vertex and $n$ edges. Recall that $C^*(E_n)={\mathcal O}_n$.

It turns out that ${\mathcal O}_n\rtimes_\rho G$ is   isomorphic to $C^*(E_n(c))$, where  $E_n(c)$ is the skew product graph $(\hat{G},\hat{G}\times E_n^1, r,s)$ with
\[r(\chi,e)=\chi c(e), s(\chi,e)=\chi\]
for $\chi\in \hat{G}$ and $e\in E^1_n$.

\begin{example}
If $G={\mathbb T}$ and $\lambda:{\mathbb T}\to U(L^2({\mathbb T}))$ is the left regular representation, then the algebra ${\mathcal O}_\infty\rtimes_\lambda \TT$ is isomorphic to the graph algebra where the vertices are labeled by ${\mathbb Z}$ and the incidence matrix has each entry equal to $1$.

\end{example}

\begin{example}
If $G={\RR}$ and $\rho :\RR\to U({\CC}^n)$ is a representation for $n\ge 2$, Kishimoto and Kumjian (see \cite{KK}) showed that ${\mathcal O}_n\rtimes_\rho {\mathbb R}$ is simple and purely infinite if the  characters in the decomposition of $\rho$ generate ${\mathbb R}$ as a closed semigroup. 

In \cite{Ka03} Katsura determined the ideal structure of the crossed products ${\mathcal O}_n\rtimes_\rho G$ where $G$ is a locally compact second countable abelian group. The action  of $G$ is determined by an $n$-dimensional representation $\rho$ with characters $\{\omega _1,...,\omega _n\}$ such that \[\rho_t(S_i)=\omega_i(t)S_i\] for $t\in G$, where $S_i$ are the generators of ${\mathcal O}_n$ for $ i=1,...,n$. It is shown in \cite{Ka03} that ${\mathcal O}_n\rtimes_\rho G\cong {\mathcal O}_{\mathcal D}$ where ${\mathcal D}$ is the  $C^*$-correspondence obtained from ${\CC}^n\otimes C^*(G)$ with the obvious right multiplication and inner product, and with left multiplication  given by 
\[f\cdot(f_1,f_2,...,f_n)=(\sigma_{\omega_1}(f)f_1,\sigma_{\omega_2}(f)f_2,...,\sigma_{\omega_n}(f)f_n),\]
for $f,f_1,...,f_n\in C_0(\hat{G})$, where $(\sigma_\omega f)(\gamma)=f(\gamma+\omega)$.

Note that in this case our $C^*$-correspondence ${\mathcal E}(\rho)$   determines a different Cuntz-Pimsner algebra, since the left multiplication is not  injective. The ideal structure of the algebras ${\mathcal O}_{\mathcal E(\rho)}$ will be considered in future work.
 
\end{example}

\end{document}